\documentclass[11pt, a4paper]{amsart} %{article}

\usepackage{geometry}
\usepackage{amsmath,amsthm,amssymb}
\usepackage{bbm, hyperref}
\usepackage[utf8]{inputenc}
\newtheorem{theorem}{Theorem}[section]
\newtheorem{lemma}[theorem]{Lemma}
\newtheorem{proposition}[theorem]{Proposition}
\newtheorem{corollary}[theorem]{Corollary}
\newtheorem{conjecture}[theorem]{Conjecture}
\newtheorem{observation}[theorem]{Observation}

\newtheorem{remark}[theorem]{Remark}
\newtheorem{example}[theorem]{Example}

\newtheorem{notation}[theorem]{Notation}
\usepackage{color}

\theoremstyle{remark}
\newtheorem*{rem}{Remark}
\newtheorem*{lem}{Lemma}

\theoremstyle{plain}

\newcommand{\J}{\text{Jac}}
\newcommand{\ld}{\text{lcm denom}}
\newcommand{\ZmZ}{\mathbb{Z}/m\mathbb{Z}}
\newcommand{\ZnZ}{\mathbb{Z}/n\mathbb{Z}}
\DeclareMathOperator{\Ima}{Im}

\DeclareMathOperator{\Div}{Div}
\DeclareMathOperator{\Val}{val}
\DeclareMathOperator{\Prin}{Prin}
\DeclareMathOperator{\Deg}{deg}
\newcommand{\nul}{\text{nul}}
\newcommand{\rk}{\text{rk}}
\newcommand{\rd}[1]{\tilde{#1}}

\newcommand{\Hom}{\operatorname{Hom}}

\begin{document}

\title{Two-Vertex Generators of Jacobians of Graphs}
\author{David Brandfonbrener, Pat Devlin, Netanel Friedenberg,\\ Yuxuan Ke, Steffen Marcus, Henry Reichard, and Ethan Sciamma}
\maketitle
\pagestyle{plain}

\begin{abstract}
We give necessary and sufficient conditions under which the Jacobian of a graph is generated by a divisor that is the difference of two vertices. This answers a question posed by Becker and Glass and allows us to prove various other propositions about the order of divisors that are the difference of two vertices. We conclude with some conjectures about these divisors on random graphs and support them with empirical evidence.
\end{abstract}

\section{Introduction}

Given a finite, undirected, connected multigraph $ G$ without loops, a \textit{divisor} is an assignment of integer values to the vertices. The \textit{degree} of a divisor is the sum of these values. The \textit{Jacobian} of $ G$, denoted $ \J(G)$, is a finite abelian group defined as the quotient of the degree zero divisors by an equivalence relation determined by chip-firing on the graph (see Section~\ref{sec:chipfiring}). The Jacobian is often cyclic, and recent work of several authors \cite{CLP14, CKLPW14, W14} considers the likelihood that a random graph has cyclic Jacobian. It has been conjectured in \cite{CKLPW14} based on a Cohen-Lenstra heuristic and empirical evidence that the probability of cyclic Jacobian in an Erd\H{o}s-R\'enyi random graph goes to $ \prod_{i=1}^\infty \zeta(2i +1 )^{-1} \approx .79$ as the number of vertices goes to infinity. Indeed, Wood in \cite{W14} proved this to be an upper bound, but no nontrivial lower bound on this probability is known.

% Maybe we should begin by stating Lorenzini's result and then giving Becker's? This paragraph might be easier to follow if we organized the results chronologically.
% I implemented this (Steffen) 
The central object of study in this paper is the divisor $ \delta_{xy}$ which is -1 at vertex $ x$, 1 at vertex $ y$, and 0 elsewhere. Since the collection of all $\delta_{xy}$ together generate the whole Jacobian, $\delta_{xy}$ presents an ideal candidate for a generator of $\J(G)$. 
% Can we say "if $\gcd(|\J(G)|, |\J(G_1)|) = 1$ instead of "if $|\J(G)|$ and $|\J(G_1)|$ are relatively prime"? This is a minor change, but it makes it a bit easier to see that Lorenzini's result is a special case of ours when comparing this paragraph with Theorem 1.1.
% I implemented this (Steffen)
In \cite[5.1]{Lor1}, Lorenzini constructs a graph $G_1$ by removing an edge $(x,y)$ from $G$ and proves that the condition $\gcd(|\J(G)|, |\J(G_1)|) = 1$ implies that the Jacobians of $G$ and $G_1$ are both cyclic. Lorenzini's proof does not establish explicit generators for these Jacobians, however. More recently, Becker and Glass \cite[Open~Question~2.8]{BnG} ask whether $\delta_{xy}$ generates $\J(G)$ under the same assumption that $\gcd(|\J(G)|, |\J(G_1)|) = 1$.

% Perhaps we can slightly reorder this paragraph as follows:
% In this paper we describe the relationship between $[\J(G): \langle\delta_{xy} \rangle]$ and $\gcd(|\J(G)|, |\J(G_1)|)$. Note that Theorem 1.1 is more general than Becker's and Lorenzini's results, which only consider the special case that $\gcd(|\J(G)|, |\J(G_1)|) = 1$. Theorem 1.1 allows us to resolve the question of \cite{BnG} with an affirmative answer and to prove a full converse  in Corollary 1.2.
% I rewrote this a bit differently than your suggestion... what do you think?

In this paper we resolve the question of \cite{BnG} with an affirmative answer, provide a stronger result, and also prove its converse. The approach we take allows us to treat the cases of adding and removing an edge between $x$ and $y$ similarly. Our work strengthens the original theorem of Lorenzini \cite[5.1]{Lor1} and the question of Becker-Glass in two ways: 1) we consider the gcd of $|\J(G)|$ and $|\J(G_1)|$ in general rather than only in the relatively prime case, and 2) we show that $\delta_{xy}$ is an explicit generator of the Jacobian under certain conditions.  

% Reading through Sam's comment on this theorem, it sounds like he initially thought that we had generalized Lorenzini's and Becker's results by proving that they also hold if you add an edge to $G$ to obtain $G_1$ (rather than removing an edge to obtain G_1). Obviously this is a misinterpretation we want to avoid. Maybe we should define $G_1$ here the same way as before -- i.e. as being G with an edge removed -- in order to prevent potential confusion?
% I'm not sure this change is needed, but make it if you want to... I don't mind (Steffen).
\begin{theorem}\label{maintheorem}
Fix vertices $ x,y $ in $ G$. Let $ G_1 $ be $ G $ with an edge added or 
removed between $ x $ and $ y$. Then
\begin{align*}
[\J(G): \langle\delta_{xy} \rangle] \ \big|\ \gcd(|\J(G)|, |\J(G_1)|)
\end{align*}
and
\begin{align*}
\gcd(|\J(G)|, |\J(G_1)|) \ \big|\ [\J(G): \langle\delta_{xy} \rangle]^2.
\end{align*}
\end{theorem}
\noindent For more general and precise statements of this theorem, see Theorems~\ref{T:GOQ2.8} and \ref{GCOQ}. Theorem~\ref{maintheorem} implies the following corollary.

\begin{corollary}\label{ifftheorem}
The divisor $\delta_{xy}$ is a generator of $ \J(G) $ if and only if\begin{align*}
\gcd(|\J(G)|, |\J(G_1)|) = 1,
\end{align*} where $ G_1$ is the graph $ G$ with an edge added or removed between $ x $ and $ y $. 
\end{corollary}

These results relate the order of $ \delta_{xy} $ to deleting and inserting edges in $ G$. Our main theorem is proven in Section~\ref{sec:maintheorem}. In Section~\ref{sec:contractions} we show that similar results hold under edge contraction. In Section~\ref{sec:bounding} we move on to considering lower bounds on the order of $ \delta_{xy}$. Finally, with the hope of eventually extending statements about $ \delta_{xy}$ to random graphs we provide some conjectures and empirical evidence in Section~\ref{sec:random} relating to the probability that some $ \delta_{xy} $ generates the Jacobian.

\subsection{Acknowledgments}
This project was completed as part of the 2017 Summer Undergraduate Mathematics Research at Yale (SUMRY) program, where the third and fifth authors were supported as mentors and the first, fourth, sixth, and seventh authors were supported as participants. We thank everyone in the program for fostering a productive research environment. We benefited from conversations with  Darren Glass, Dave Jensen, and Dhruv Ranganathan. We thank Sam Payne for suggesting the problem. 

\section{Background}
\subsection{Chip-firing and the Jacobian of graphs}\label{sec:chipfiring}

Unless specified otherwise, throughout this paper we will assume that any graph $G$ is a connected, undirected multigraph without loops. Let $V(G)$ and $E(G)$ denote the vertex and edge sets of $G$ respectively, and let $n = |V(G)|$ be the number of vertices. For a vertex $v\in V(G)$, let $\Val(v)$ denote the valency of $v$ --- i.e. the number of edges incident to $v$.

Following \cite{baknor}, a \emph{divisor} $D \in \mathbb{Z}^n$ on $G$ is a formal $\mathbb{Z}$-linear combination of the vertices of $G$. Divisors are often interpreted as assignments of an integer number of ``chips'' to each of the vertices of $G$. The \emph{degree} of a divisor $\Deg(D) = \sum_{v \in V(G)}D(v)$ is the divisor's total number of chips. $\Div(G)$ denotes the group of all divisors on $G$, and $\Div^0(G)$ denotes the subgroup of divisors with degree zero.

There is an equivalence relation on the elements of $\Div(G)$ based on the well-known \emph{chip-firing game}, which is defined as follows. Given any divisor $D \in \Div(G)$, let $K_{v,w}$ denote the number of edges between two vertices $v,w \in V(G)$. We can \emph{fire} a vertex $v$ by sending one chip from $v$ along each of its incident edges to adjacent vertices. This operation produces a new divisor $D'$, given by:
\[
	D'(w) =
    \begin{cases}
    	D(v) - \Val(v) &\text{if }  w=v \\
        D(w) + K_{v,w} &\text{if w is adjacent to v} \\
        D(w)  &\text{otherwise}.
    \end{cases}
\]
Note that any chip-firing move at a given vertex can be reversed by firing at all of the other vertices. We define a \emph{firing script} to be a vector $\sigma \in \mathbb{Z}^n$ whose entries specify the number of times each vertex in $G$ should be fired. We say that two divisors $D_1$ and $D_2$ are \emph{equivalent} if there exists a firing script taking one to the other. We denote this relation by $\sim$. Note that the degree of a divisor is invariant under chip-firing. The set of \emph{principal divisors} $\Prin(G)$ is the set of divisors equivalent to the divisor with zero chips on every vertex, which we will call the \emph{zero divisor}.

The \emph{Jacobian} $\J(G) := \Div^0(G)/\Prin(G)$ is the group of equivalence classes of divisors on $G$ with degree zero. The Jacobian is sometimes also referred to as the \emph{sandpile group} or the \emph{critical group}. If $D \in \Div^0(G)$, we denote the equivalence class containing $D$ by $[D]$. $\J(G)$ is always finite for a connected graph, and its order is equal to the number of spanning trees on $G$ by the matrix-tree theorem (see \cite{baknor}, or \cite{CorPerk}). 

The Jacobian of a graph is equal to the torsion subgroup of the cokernel of the graph's \emph{Laplacian}, which is an $n \times n$ integer matrix denoted by $L$ and defined as follows. Let $\Delta$ be the diagonal matrix with $(i,i)$--entry equal to $\Val(v_i)$, and let $A$ be the adjacency matrix of $G$. Then $L = \Delta - A$. $L$ gives a map from the the space of firing scripts to the space of principal divisors: if $\sigma$ is a firing script, then $L\sigma$ is the principal divisor obtained by applying $\sigma$ to the zero divisor (again, see \cite{baknor} or \cite{CorPerk}). 

For a given choice of $i \in \{ 1,...,n\}$ the \emph{reduced Laplacian} $\tilde{L}$ is $L$ with the $i$th row and $i$th column removed. $\tilde{L}$ is invertible if $G$ is connected, regardless of the choice of $i$, and furthermore we have $\det(\tilde{L}) = |\J(G)|$. 

In similar fashion, if $D \in \Div^0(G) \subset \mathbb{Z}^n$ then we define the \emph{reduced divisor} $\tilde{D} \in \mathbb{Z}^{n-1}$ to be $D$ with the $i$th entry deleted. If $D \in \Div^0(G)$, then $D$ is uniquely specified by $n-1$ of its entries, since the final entry must be the negative sum of the other $n-1$ entries. 

Likewise, since the all-ones vector generates the kernel of $L$, for any firing script $\sigma$ there is a unique firing script $\sigma_0$ such that the $i$th entry of $\sigma_0$ is zero and $L\sigma = L\sigma_0$. We define the \emph{reduced firing script} $\tilde{\sigma} \in \mathbb{Z}^{n-1}$ of $\sigma$ to be $\sigma_0$ with the $i$th entry deleted. Unless specified otherwise, we will always let $i=n$ be the index deleted when referring to a reduced Laplacian, divisor, or firing script. Note that $ \widetilde{L\sigma} = \tilde{L}\tilde{\sigma}$ as expected.
 
\begin{notation} \label{notat}
Throughout this paper, we will let $m$ denote $|\J(G)|$. There is a bijection $\phi$ between the degree zero divisors and the reduced divisors given by $\phi(D) = \tilde{D}$ and a bijection $\rho$ between the firing scripts mod the all-ones vector and the reduced firing scripts given by $\rho(\sigma) = \tilde{\sigma}$.
\end{notation}

\subsection{Monodromy weights}
In \cite{Shok}, Shokrieh considers a symmetric, bilinear map from $\J(G)\times\J(G)$ to $\mathbb{Q}/\mathbb{Z}$ known as the monodromy pairing. To prove our main results, we will utilize a related map from $\J(G)$ to $\ZmZ$ that we will call monodromy weights. The relationship between monodromy weights and Shokrieh's monodromy pairing will be made explicit in section \ref{MonPair}.

Any homomorphism $\phi: \text{Div}^0(G) \to \ZmZ$ has to take the form of a dot product, thus for $D \in \Div^0(G)$ we can write $\phi(D) = w \cdot D \pmod{m}$, where $w \in \mathbb{Z}^n$. The entries of $w$ correspond to an assignment of integer weights to vertices of $G$. For any given homomorphism $\phi$, there are many possible weight vectors $w$ representing it: we can add any multiple of $m$ to the weight on a given vertex without changing $\phi$, and we can also add a constant vector to $w$ without changing $\phi$.

While any arbitrary weight vector $w \in \mathbb{Z}^n$ represents a homomorphism $\phi : \text{Div}^0(G) \to \ZmZ$, this homomorphism descends to a well-defined map from $\J(G)$ to $\ZmZ$ if and only if $w \cdot D \equiv 0 \pmod{m}$ for all $D \in \Prin(G)$. We call such a weight vector a \emph{monodromy weight} on $G$. The following proposition provides a method of finding monodromy weights.

\begin{proposition}\label{P2.1}
Given a graph $G$ with Laplacian $L$, a vector $w$ satisfies $Lw \equiv \boldmath{0} \pmod{m}$ if and only if it is a monodromy weight.
\end{proposition}

\begin{proof}
Let $\sigma$ be a firing script and let $D_0 = L\sigma$ be the corresponding principal divisor. For $w$ to be a monodromy weight, we require that $w \cdot D_0 = w \cdot L\sigma \equiv \boldmath{0} \pmod{m}$. Since $L$ is symmetric, this is equivalent to saying $Lw\cdot \sigma \equiv \boldmath{0} \pmod{m}$. Since $\sigma$ is arbitrary, this is the same as requiring that $Lw \equiv \boldmath{0} \pmod{m}$. Thus the solutions of $Lw \equiv \boldmath{0} \pmod{m}$ are precisely the monodromy weights on $G$.
\end{proof}

Proposition \ref{P2.1} gives us an idea of how to find monodromy weights. However, actually solving the equation is not as simple, since $L$ is singular over the integers. Furthermore, infinitely many monodromy weights represent the same homomorphism from $\J(G)$ to $\ZmZ$; we would like to identify a set of representatives among the monodromy weights, one for each homomorphism. 
%Furthermore, there are infinitely many possible monodromy weights, but many of them represent the same homomorphism from $\J(G)$ to $\ZmZ$ and consequently are redundant; we would like to identify the set of monodromy weights representing distinct homomorphisms. 

If $w$ is a monodromy weight representing a homomorphism $\phi$, then by adding a constant vector to $w$ we can make its $n$th entry zero without changing the homomorphism $w$ represents, obtaining a new monodromy weight $w_0$. Let $\tilde{w}$ denote the first $n-1$ entries of $w_0$; we call $\tilde{w}$ a \emph{reduced monodromy weight} representing $\phi$. The following two propositions address the first of the problems mentioned above, providing a means of solving for the monodromy weights.

\begin{proposition} \label{P2.2}
Let $w \in \mathbb{Z}^n$ be a weight vector. Then $w$ is a monodromy weight if and only if $\tilde{L}\tilde{w} \equiv \boldmath{0} \pmod{m}$.
\end{proposition}

\begin{proof}
Observe that since $w$ and $w_0$ represent the same homomorphism, $Lw \equiv \boldmath{0} \pmod{m}$ if and only if $Lw_0 \equiv \boldmath{0} \pmod{m}$. The first $n-1$ entries of $Lw_0$ are simply $\tilde{L}\tilde{w_0}$. The last row of $L$ is the negative sum of the first $n-1$ rows by the definition of the Laplacian; hence, if the dot product of $w_0$ with each of the first $n-1$ rows of $L$ is zero modulo $m$, the dot product of $w_0$ with the $n$th row must also be zero, implying that $Lw_0 \equiv \boldmath{0} \pmod{m}$. Thus we see that $Lw_0 \equiv \boldmath{0} \pmod{m}$ if and only if $\tilde{L}\tilde{w_0} \equiv \boldmath{0} \pmod{m}$. 
\end{proof}

\begin{proposition} \label{FindWeight} Any reduced monodromy weight $\tilde{w}$ can be found by solving the equation $\tilde{L}\tilde{w} = m\tilde{D}$ for $\tilde{w}$ over the integers for some $D\in\Div^{0}(G)$.
\end{proposition}

\begin{proof}
First note that $\tilde{L}$ is invertible over the rationals, so there is a unique solution $\tilde{w} = m\tilde{L}^{-1}\tilde{D}$ for $\tilde{w}$. We have $\tilde{L}\tilde{w} = \tilde{L} m \tilde{L}^{-1} \tilde{D} = m \tilde{D} \equiv \boldmath{0} \pmod{m}$, so by Proposition \ref{P2.2} $\tilde{w}$ is a reduced monodromy weight.
\end{proof}

Proposition \ref{FindWeight} gives us a concrete way of finding reduced monodromy weights. We now address the second problem mentioned earlier by showing that if we consider reduced monodromy weights modulo m, they correspond bijectively with the elements of $\J(G)$.

\begin{proposition}\label{LEquiv}
Let $D_1, D_2 \in \Div^0(G)$ and let $\tilde{w_1},\tilde{w_2}$ be derived from $D_1, D_2$ as described in Proposition \ref{FindWeight}. Then $\tilde{w_1} \equiv \tilde{w_2} \pmod{m}$ if and only if $D_1 \sim D_2$.
\end{proposition}

\begin{proof}
If $D_1 \sim D_2$, then $\tilde{D}_1 = \tilde{L}\tilde{\sigma} + \tilde{D}_2$ for some reduced firing script $\tilde{\sigma}$.  It is easy to check that $\tilde{w}_1 = \tilde{w}_2 + m\tilde{\sigma}$, which gives $\tilde{w_1} \equiv \tilde{w_2} \pmod{m}$. Conversely, if $\tilde{w_1} \equiv \tilde{w_2} \pmod{m}$ then $\tilde{w}_1 = \tilde{w}_2 + mv$ for some $v \in \mathbb{Z}^{n-1}$. By the derivation of $\tilde{w_1}$ and $\tilde{w_2}$, this means that $m\tilde{L}^{-1}(\tilde{D}_1 - \tilde{D}_2) = mv$, so we have $\tilde{D}_1 - \tilde{D}_2 = \tilde{L}v$, which implies that $D_1 \sim D_2$.
\end{proof}

Let $\Hom(\J(G),\ZmZ)$ be the group of homomorphisms $\phi: \J(G) \to \ZmZ$, let $K$ be the group of reduced monodromy weights taken modulo $m$, and let $C^{\tilde{L}}=m\tilde{L}^{-1}$ be the cofactor matrix of the reduced Laplacian. Propositions \ref{FindWeight} and \ref{LEquiv} tell us that there is an isomorphsim $F : \J(G) \to K$ given by $[D]\mapsto C^{\tilde{L}}D\pmod{m}$, which suggests that $\Hom(\J(G),\ZmZ)$ may also be isomorphic to $\J(G)$. This is indeed the case, and is in fact a manifestation of a far more general and well-known result: any finite abelian group is (non-canonically) isomorphic to its Pontryagin dual. Given $\tilde{w}\in K$, let $\phi_{\tilde{w}}\in \Hom(\J(G),\ZmZ)$ be the associated homomorphism mapping $[D]\mapsto \tilde{w}\cdot \tilde{D}\pmod{m}$; let $\Phi:K\to \Hom(\J(G),\ZmZ)$ be the map $\tilde{w}\mapsto\phi_{\tilde{w}}$. The earlier discussion shows that $\Phi$ is surjective.
%Let $H$ be the group of homomorphisms $\phi: \J(G) \to \ZmZ$ and let $K$ be the group of reduced monodromy weights taken modulo $m$. Proposition \ref{LEquiv} implies that there is an isomorphsim $F : \J(G) \to K$, which suggests that $H$ may also be isomorphic to the $\J(G)$. This is indeed the case, and is in fact a manifestation of a far more general and well-known result: any finite abelian group is (non-canonically) isomorphic to its Pontryagin dual. Let $C^{\tilde{L}}=m\tilde{L}^{-1}$ be the cofactor matrix of the reduced Laplacian, and let $T : \J(G) \to H$ map an element $[D] \in \J(G)$ to the homomorphism $\phi \in H$ represented by the reduced monodromy weight $C^{\tilde{L}}\tilde{D}$. Then we obtain the following:

\begin{proposition} \label{CanIsom}
The composition $\Phi\circ F:\J(G)\to K\to\Hom(\J(G),\ZmZ)$ is a canonical isomorphism between $\J(G)$ and $\Hom(\J(G),\ZmZ)$.
%The composition \comdl{\J(G)\ar[r]^(.6){F}&K\ar[r]^(.225){\Phi}&\Hom(\J(G),\ZmZ)} is a canonical isomorphism between $\J(G)$ and $\Hom(\J(G),\ZmZ)$.
%$T$ is a canonical isomorphism between $\J(G)$ and $H$.
\end{proposition}

\begin{proof}
It remains only to show that $\Phi$ is injective. Suppose that $\phi_{\tilde{w}_1}=\phi_{\tilde{w}_2}$. Then $(\tilde{w}_1 - \tilde{w}_2)\cdot \tilde{D} \equiv 0 \pmod{m}$ for all $\tilde{D} \in \mathbb{Z}^{n-1}$. Since $\tilde{D}$ can be chosen arbitrarily, this implies that $\tilde{w}_1 - \tilde{w}_2 \equiv 0 \pmod{m}$.
%Since $\J(G) \cong K$, it suffices to show that each $\tilde{w} \in K$ corresponds to a distinct $\phi \in H$. Suppose that $\tilde{w}_1$ and $\tilde{w}_2$ represent the same homomorphism. Then $(\tilde{w}_1 - \tilde{w}_2)\cdot \tilde{D} \equiv 0 \pmod{m}$ for all $\tilde{D} \in \mathbb{Z}^{n-1}$. Since $\tilde{D}$ can be chosen arbitrarily, this implies that $\tilde{w}_1 - \tilde{w}_2 \equiv 0 \pmod{m}$.
\end{proof}

%\begin{proof}
%Proposition \ref{FindWeight} shows that $C^{\tilde{L}}$ induces a map $f$ between $\J(G)$ and $K$ given by $f([D]) \mapsto C^{\tilde{L}}D \pmod{m}$. By proposition \ref{LEquiv}, f is an isomorphism. Since $\J(G)$ is self-dual, $|H| = |\J(G)|$, and $|\J(G)| = |K|$ because $\J(G) \cong K$. Consequently, we deduce that each homomorphism $\phi \in H$ is represented by a unique element of $K$. Since $\J(G) \cong K$, this implies that T is an isomorphism.
%\end{proof}

Note that the uniqueness of the reduced monodromy weight in $K$ representing a given homomorphism $\phi \in \Hom(\J(G),\ZmZ)$ implies that any two monodromy weights represent the same homomorphism if and only if their difference taken modulo $m$ is a constant vector. Proposition \ref{CanIsom} can also be used to obtain information about the cardinality of a minimum generating set for $\J(G)$, which we will denote as $\rk(\J(G))$. Let $\rk_m(A)$ denote the rank of $A$ over $\ZmZ$ and let $\nul_m(A)$ denote the nullity of $A$ over $\ZmZ$. Then we can derive the following corollary.

\begin{corollary} \label{rankcor}
Over $\ZmZ$, we have $\nul_m(\tilde{L}) = \rk_m(C^{\tilde{L}}) = \rk(\J(G))$.
\end{corollary}

\begin{proof}
Proposition \ref{CanIsom} implies that $\rk_m(C^{\tilde{L}}) = \rk(\J(G))$. If $\tilde{u} \in \ker\tilde{L}$, then by Proposition \ref{P2.2} $\tilde{u} \in K$, so by Proposition \ref{CanIsom} $\tilde{u} \in \Ima C^{\tilde{L}}$. Conversely, if $\tilde{u} \in \Ima C^{\tilde{L}}$, then over $\mathbb{Z}^{n-1}$ we have $\tilde{u} + m\tilde{v} = m\tilde{L}^{-1}\tilde{w}$ for some $\tilde{v},\tilde{w} \in \mathbb{Z}^{n-1}$. Thus $\tilde{L}\tilde{u} = m(\tilde{w} - \tilde{L}\tilde{v}) = 0 \pmod{m}$, so $\tilde{u} \in \ker\tilde{L}$. Consequently, we have $\ker\tilde{L} = \Ima C^{\tilde{L}}$, which implies that $\nul_m(\tilde{L}) = \rk_m(C^{\tilde{L}})$.
\end{proof}

\begin{example}
Let $G$ be the cycle graph on six vertices. Note that $\J(G) \cong \mathbb{Z}/6\mathbb{Z}$, so $m=6$. $\tilde{L}$ and $C^{\tilde{L}}$ are the following over $\mathbb{Z}/6\mathbb{Z}$:
\[
\tilde{L} = 
\begin{bmatrix}
2 & -1 & 0 & 0 & 0 \\
-1 & 2 & -1 & 0 & 0 \\
0 & -1 & 2 & -1 & 0 \\
0 & 0 & -1 & 2 & -1 \\
0 & 0 & 0 & -1 & 2
\end{bmatrix}
\qquad
C^{\tilde{L}} = 
\begin{bmatrix}
5 & 4 & 3 & 2 & 1 \\
4 & 2 & 0 & 4 & 2 \\
3 & 0 & 3 & 0 & 3 \\
2 & 4 & 0 & 2 & 4 \\
1 & 2 & 3 & 4 & 5
\end{bmatrix}
\]
Let $\text{col}_i$ denote the $ith$ column of $C^{\tilde{L}}$. Observe that modulo 6, we have $\text{col}_i = i \cdot\text{col}_1$, which implies that $\rk_m(C^{\tilde{L}}) = 1$. Since $\J(G)$ is cyclic, it has a generating set of cardinality 1, so likewise $\rk(\J(G)) = 1$. Finally, it can easily be verified that the kernel of $\tilde{L}$ over $\mathbb{Z}/6\mathbb{Z}$ is generated by the vector $v = (1,2,3,4,5)$, so $\nul_m(\tilde{L}) = 1$ as well.
\end{example}

\subsection{Relationship to the monodromy pairing} \label{MonPair}

We now make the relationship between monodromy weights and the monodromy pairing explicit. Let $M$ be any generalized inverse of the Laplacian matrix. Shokrieh's pairing is given by $\langle \cdotp , \cdotp \rangle :\J(G) \times \J(G) \to \mathbb{Q}/\mathbb{Z}$, evaluated as follows:
\[
\langle D_1 , D_2 \rangle = D_1^T M D_2.
\]
Using the isomorphism $\Phi\circ F:\J(G)\to \Hom(\J(G),\ZmZ)$ we can define a map $\varphi : \J(G) \times \J(G) \to \ZmZ$ by $([D_1],[D_2])\mapsto \Phi\circ F([D_1])([D_2])$. The map can be computed as follows. Let $D_1, D_2$ be representatives of elements of $\J(G)$. Then $\varphi([D_1],[D_2])$ is given by
%For any homomorphism $\phi \in H$, the reduced monodromy weight $\tilde{w}$ representing $\phi$ is fixed. However if we allow $\tilde{w}$ to be a function on $\J(G)$, we can define a map $\varphi : \J(G) \times \J(G) \to \ZmZ$. The map can be computed as follows. Let $D_1, D_2$ be representatives of elements of $\J(G)$. Then let
\[
\rd{D}_1^T C^{\tilde{L}} \rd{D}_2 \pmod{m}.
%\varphi(D_1,D_2) = D_1^T C^{\tilde{L}} D_2
\]
In this computation, the reduction modulo $m$ of $C^{\tilde{L}}\rd{D}_2$ is $F([D_2])$, and multiplying by $\rd{D}_1^T$, i.e.\ taking the dot product with $\rd{D}_1$, is then just the application of the corresponding homomorphism to $[D_1]$.

This map is, in essence, the same as the map defined by Shokrieh, differing only by a multiplicative factor of $m$ and the use of any generalized inverse in the monodromy pairing instead of $\tilde{L}^{-1}$. Since the matrix obtained from $\tilde{L}^{-1}$ by making the $n$th row and columns all $0$ is a generalized inverse of $L$, we see that the two pairings are really the same. 
% In light of proposition \ref{CanIsom}, we can now understand both the monodromy pairing and the monodromy weights as a group action of the Jacobian on itself, acting as a group homomorphism into $\ZmZ$ or $\mathbb{Q}/\mathbb{Z}$.

\section{Edge deletion and insertion} \label{sec:maintheorem}

In this section, we provide the proofs of the two main theorems from the introduction. First we provide a lemma relating the order of a divisor to the monodromy weights.

\begin{lemma}\label{GTwsur}
Let $D \in \Div^0(G)$ be a divisor, $\rd{w} = C^{\tilde{L}} \rd{D} \in \mathbb{Z}^{n-1}$ with each component taking a value at least 0 and less than $ m$.
Let $\phi: \Div^0(G) \to \ZmZ$ be the map induced by inner product with $\rd{w}$. Then
%Let $\phi: \Div^0(G) \to \ZmZ$ be the map induced by inner product with $w$. Then
\begin{align*}
|[D]|_{\J(G)} = |\text{Im}(\phi)| = m/\gcd(m, \rd{w})
\end{align*}
where $\gcd(m,\rd{w})$ denotes the gcd of $m$ and the entries of the vector $\rd{w}$, and $|\cdot|_{\J(G)}$ gives the order of an element of $ \J(G)$.
%where $\gcd(m,w)$ denotes the gcd of $m$ and the entries of the vector $\rd{w}$, and $|\cdot|_{\J(G)}$ gives the order of an element of $ \J(G)$.
\end{lemma}

\begin{proof} 
The order of $[D]$ is the smallest integer $k$ such that the divisor $kD$ is linearly equivalent to the zero divisor $\textbf{0}$. Thus we have $ kD - L\sigma = \textbf{0}$ for some firing script $\sigma$. Thus $\tilde{L}\rd{\sigma} = k\rd{D}$, and multiplying both sides by $\tilde{L}^{-1}$ gives $k\tilde{L}^{-1}\rd{D} = \rd{\sigma}$. Since $\sigma$ is a firing script, it must be an integer vector, thus $k$ is the smallest integer such that $k\tilde{L}^{-1}\rd{D}$ is an integer vector.

Then since $k$ is minimal, the gcd of the entries of $\rd{\sigma}$ and $k$ is $1$. Since $\rd{w} = C^{\tilde{L}} \rd{D}$, we can write $\rd{w} = (m/k)\rd{\sigma}$, and thus $\gcd(m,\rd{w}) = m/k$. So, $\text{Im}(\phi)$ is a subgroup of $\ZmZ$ generated by $(m/k)$ since we can generate any multiple of $ m/k$ via inner product with $ w$ by the Euclidean algorithm. Therefore, $|\text{Im}(\phi)| = m/(m/k) = k = m/\gcd(m,\rd{w})$.
\end{proof}

The following statement is implied.

\begin{corollary}\label{orderCor}
Let $D \in \Div^0(G)$, then
\begin{align*}
|[D]|_{\J(G)} = \det{\tilde{L}}/\gcd(C^{\tilde{L}}\rd{D}, \det{\tilde{L}})
\end{align*}
where $\gcd(C^{\tilde{L}}\rd{D}, \det{\tilde{L}})$ represents the greatest common divisor of the components of the vector $C^{\tilde{L}}\rd{D}$ and $\det{\tilde{L}}$.
\end{corollary}

%consider cutting this out entirely
%Lemma \ref{GTwsur} implies the following corollary for the case when $ [D] $ is a generator of $ \J(G)$. 

%\begin{corollary}\label{Twsur}
%Suppose that $\J(G)$ is cyclic and $[D]$ is a generator of $\J(G)$. Let $\rd{w} = C^{\tilde{L}} \rd{D}$, where $m = \det{\tilde{L}}$ as before. Then the map $\phi: \text{Div}^0(G) \to \ZmZ$ is surjective. \Steffen{do we need this Corollary?}
%\end{corollary}

% \begin{theorem}\label{T:OQ2.8}
% Let $G$ be a connected simple graph, and $G_1$ the graph obtained by deleting a single edge between the vertices $x$ and $y$ such that $|\J(G)|$ and $|\J(G_1)|$ are relatively prime. Then $\delta_{xy} = x - y$ (the divisor that is 1 at $ x $, -1 at $ y $, and 0 elsewhere) generates $\J(G)$ and, in particular, $\J(G)$ is cyclic.
% \end{theorem}
% Theorem \ref{T:OQ2.8} is a special case of the following generalization, to which we offer a proof. But first, we prove one more lemma.

Before proving our first main theorem, we offer one last lemma.

\begin{lemma}\label{Cyylem}
Let $ G $ be a multigraph with an edge between vertices $ x $ and $ y$, and let $ \tilde{L} $ be the Laplacian of $ G $ reduced by the vertex $ x$. Then $ C^{\tilde{L}}_{yy} $ is the number of spanning trees in $ G $ that include the edge $ (x,y)$. 
\end{lemma}
\begin{proof}
Let $ G_1 $ the graph obtained by deleting an edge between the vertices $ x $ and $ y$. Let $ x $ be the deleted row and column in our definition of the reduced Laplacian of both $ G $ denoted $ \tilde{L}$ and $ G_1$ denoted $\tilde{L}_1$. Note that the only difference between the two matrices is that $ (\tilde{L}_1)_{yy} = \tilde{L}_{yy} - 1$. Now we note that writing the determinants of the reduced Laplacians in terms of their cofactors expanded along the $ y $ column, we have that:
	\begin{align*}
		\det \tilde{L}_1 &= \sum_{w \in V(G) \backslash \{x\}} (\tilde{L}_1)_{wy} C_{wy}^{\tilde{L_1}} = (\tilde{L}_1)_{yy}C_{yy}^{\tilde{L}_1} +  \sum_{w \in V(G) \backslash \{x,y\}} (\tilde{L}_1)_{wy} C_{wy}^{\tilde{L_1}} \\ &= (\tilde{L}_{yy} - 1)C_{yy}^{\tilde{L}} + \sum_{w \in V(G) \backslash \{x\}} \tilde{L}_{wy} C_{wy}^{\tilde{L}} \\ &= \det \tilde{L} -  C_{yy}^{\tilde{L}}
	\end{align*}

Recall that $ \det \tilde{L}$ equals the number of spanning trees of $G$. Furthermore, the number of spanning trees on $ G $ is the sum of the number of spanning trees including $ (x,y) $ and the number of spanning trees without $ (x,y)$. Since $ \det \tilde{L}_1 $ is the number of spanning trees of $ G $ without $ (x,y)$, the above equation gives us the desired result that $ C_{yy}^{\tilde{L}}$ is the number of spanning trees including $ (x,y)$. 
\end{proof}

We can use these lemmas to answer the question of when a divisor supported on only two vertices is a generator. The following theorem provides and generalizes an affirmative answer to a question by Becker and Glass \cite[Open~Question~2.8]{BnG} based off of a theorem of Lorenzini \cite[5.1]{Lor1}.

\begin{theorem}\label{T:GOQ2.8}
Let $G$ be a connected multigraph, and $G_1$ the multigraph obtained by deleting any integer $ k_{xy} $ of the edges between the vertices $x$ and $y$ (where negative $ k_{xy}$ represents adding edges). Let $S$ be the subgroup  of $\J(G)$ generated by $[\delta_{xy}]$. Then 
\begin{align*}
[\J(G): S] \ \big\vert\  gcd(|\J(G)|,|\J(G_1)|).
\end{align*} Let $S_1$ be the subgroup  of $\J(G_1)$ generated by $[\delta_{xy}]$. Then 
\begin{align*}
[\J(G): S_1]\ \big\vert\ gcd(|\J(G)|,|\J(G_1)|).
\end{align*}
\end{theorem}

\begin{proof}
	Recall that $ m = |\J(G)|$. As in Lemma \ref{Cyylem}, let $ x $ correspond to the row and column in our definition of the reduced Laplacian of both $ G $ denoted $ \tilde{L}$ and $ G_1$ denoted $\tilde{L}_1$. By Lemma \ref{Cyylem}, we have that $ C_{yy}^{\tilde{L}}$ is the number of spanning trees using a specific $ x,y $ edge. So, $k_{xy} 				C_{yy}^{\tilde{L}}$ gives the number of spanning trees using any of the $ k_{xy} $ edges that are removed to form $ G_1$. Since the number of spanning trees in $ G$ must be the sum of the number of spanning trees in the graph without $ k_{xy} $ edges plus the number of trees including those edges, we get by the matrix tree theorem that:
    \begin{align*}
    	\det \tilde{L} = \det \tilde{L}_1 + k_{xy} 				C_{yy}^{\tilde{L}}.
    \end{align*}
    
    So we let $ gcd(|\J(G)|,|\J(G_1)|) = \gcd(\det \tilde{L} , \det \tilde{L}_1 ) $, so that both
	\begin{align*}
		&\gcd(\det{\tilde{L}}, C_{yy}^{\tilde{L}}) \ \big|\  gcd(|\J(G)|,|\J(G_1)|) \\ \ \ \text{and} \ \ &\gcd(\det{\tilde{L}_1}, C_{yy}^{\tilde{L}}) \ \big|\ gcd(|\J(G)|,|\J(G_1)|).
	\end{align*}
	Note that both of these statements would be equality if $ k_{xy} = \pm 1$, which is guaranteed if we force $ G $ to be simple.
	
	Now let $w = C^{\tilde{L}}\rd{\delta}_{xy} = C_{yy}^{\tilde{L}}\rd{\delta}_{xy}$, noting that the reduced $ \delta_{xy}$ is just the indicator vector of vertex $ y$. Then, by Corollary \ref{orderCor} we have that 
	\begin{align*}
		|[\delta_{xy}]|_{\J(G)} = \det{\tilde{L}} / \gcd(\det{\tilde{L}}, C^{\tilde{L}} \rd{\delta}_{xy}),
	\end{align*}
	and similarly that 
	\begin{align*}
		|[\delta_{xy}]|_{\J(G_1)} = \det{\tilde{L}_1} / \gcd(\det{\tilde{L}_1}, C^{\tilde{L}_1} \rd{\delta}_{xy}).
	\end{align*}
	
	Putting these facts together, we get that
	\begin{align*}
		[\J(G): S] &= \frac{det \tilde{L}}{|[\delta_{xy}]|_{\J(G)}} = \gcd(\det{\tilde{L}}, C^{\tilde{L}}\rd{\delta}_{xy}) \ \big|\ \gcd(\det{\tilde{L}}, C_{yy}^{\tilde{L}})\\ &\big|\   gcd(|\J(G)|,|\J(G_1)|),
	\end{align*}
	and similarly
	\begin{align*}
		[\J(G_1): S_1]\ \big|\ gcd(|\J(G)|,|\J(G_1)|).
	\end{align*}
	
\end{proof}

Following the statement of this theorem, it is natural to ask if the converse is also true. That is, if $[\delta_{xy}]$ is a generator of the Jacobian, are the orders of the Jacobians with and without an edge necessarily relatively prime? The following theorem answers this question.

\begin{theorem}\label{GCOQ} %Generalized Converse of Open Question
Let $G$ be a connected multigraph, and $G_1$ the graph obtained by removing $ k_{xy}$ edges $(x,y)$ where negative $ k_{xy} $ is adding edges. Moreover, assume that $ \gcd(|\J(G)|,k_{xy}) = 1$. Let $S$ be the subgroup of $\J(G)$ generated by $[\delta_{xy}]$. Then 
\begin{align*}
\gcd(|\J(G)|,|\J(G_1)|)\ \big\vert\ [\J(G):S]^2.
\end{align*} 
Similarly, let $S_1$ be the subgroup of $\J(G_1)$ generated by $[\delta_{xy}]$. Then 
\begin{align*}
\gcd(|\J(G)|,|\J(G_1)|)\ \big\vert\ [\J(G_1):S_1]^2.
\end{align*}
\end{theorem}
\begin{proof}
As in Theorem \ref{T:GOQ2.8} we choose $ \tilde{L} $ to be the reduced Laplacian, reducing at vertex $ x$.

Note that since $\gcd(m,k_{xy}) = 1$ and $\det \tilde{L} = \det \tilde{L}_1 + k_{xy} C_{yy}^{\tilde{L}} $ as in the proof of the previous theorem, we have that 
\begin{align}\label{newline}
\gcd(|\J(G)|,|\J(G_1)|) = \gcd(m,|C^{\tilde{L}}_{yy}|).  
\end{align}

Let $\tilde{w} = C^{\tilde{L}}\rd{\delta}_{xy}$, where $C^{\tilde{L}}$ is the cofactor matrix of the reduced Laplacian. Then let $\phi:\J(G) \to \ZmZ$ be the map induced by inner product with $ w$. Now, note that $ \phi $ is a group homomorphism, so $m = |\text{Im}(\phi)| \cdot |\ker(\phi)|$. So, we also have that
\begin{align}\label{tgline1}
|\ker(\phi)|  = m / |\text{Im}(\phi)|.
\end{align}

Now, let $\phi_{S}: S \to \ZmZ$ be the map $\phi$ restricted to the subgroup generated by $[\delta_{xy}]$. Note that since $w = C^{\tilde{L}}\rd{\delta}_{xy}$, we have that the weight on the $y$th vertex is exactly $C^{\tilde{L}}_{yy}$. Thus the image of $S$ (the subgroup consisting of multiples of $\delta_{xy}$, ie multiples with weight only on $ y $ in the reduced Laplacian) under the map $\phi$ is the subgroup consisting of multiples of $\gcd(m,|C^{\tilde{L}}_{yy}|)$ in $\ZmZ$. The order of this subgroup is $m/\gcd(m,|C^{\tilde{L}}_{yy}|)$, so $|\text{Im}(\phi_{S})| = m/\gcd(m,|C^{\tilde{L}}_{yy}|)$, and thus 
\begin{align}\label{tgline2}
\gcd(m,|C^{\tilde{L}}_{yy}|) = m/|\text{Im}(\phi_{S})|.
\end{align}

Observe that this is similar to but not the same as Lemma \ref{GTwsur}, since now we have $ \gcd(m,|C^{\tilde{L}}_{yy}|)$ in the denominator rather than $ \gcd(m,\tilde{w})$. 

Since $ \varphi $ is also a homomorphism, we have $ |S| = |\text{Im}(\phi_{S})| \cdot |\ker(\phi_{S})|$. This gives us that 
\begin{align}\label{tgline3}
|\text{Im}(\phi_{S})|  = |S| / |\ker(\phi_{S})|.
\end{align}

The last fact we need is that $ \ker(\phi_{S})$ is a subgroup of $ \ker(\phi)$. Thus, we have that for some integer $c$ that 
\begin{align}\label{ttgline4}
|\ker(\phi_{S})| = |\ker(\phi)| / c. 
\end{align}

Now we have that by (\ref{newline}), (\ref{tgline2}), and (\ref{tgline3})
\begin{align*}
	\gcd(|\J(G)|,|\J(G_1)|) &= \gcd(m,|C^{\tilde{L}}_{yy}|) = \frac{m}{|\text{Im}(\phi_{S})|} = \frac{m}{|S| / |\ker(\phi_{S})|}.
\end{align*}
By lemma \ref{GTwsur} we have that $ |S| = |\text{Im}(\phi)|$. So, with (\ref{ttgline4}) and (\ref{tgline1}) we conclude that 
\begin{align*}
\gcd(|\J(G)|,|\J(G_1)|) &= \frac{m}{|\text{Im}(\phi)| / |\ker(\phi_{S})|} = \frac{m}{|\text{Im}(\phi)|} |\ker(\phi_{S})|\\ & = \frac{m}{|\text{Im}(\phi)|} \left(\frac{|\ker(\phi)|}{c}\right) = \frac{m}{|\text{Im}(\phi)|} \left(\frac{m/|\text{Im}(\phi)|}{c}\right)
	\\ &= \frac{(m/|\text{Im}(\phi)|)^2}{c} = \frac{[\J(G): S]^2}{c}.
\end{align*} 
This proves the theorem for $ G $ and $ S $, and in addition, the multiplicative factor is the index of $\ker(\phi|_{S})$ in $\ker(\phi)$. The proof for $G_1$ and $ S_1 $ follows immediately since we allowed removing or adding edges.
\end{proof} 
	
Note that we have proved Theorem \ref{GCOQ} for multigraphs where $ k_{xy}$ is relatively prime to $ m$. If we have a simple graph, then $ k_{xy} = \pm 1$ is relatively prime with $ m$ immediately, so we get the same result. From the above theorems we can now give a precise condition for when $ [\delta_{xy}] $ is a generator of $\J(G)$ by considering the case when $ \gcd(|\J(G)|,|\J(G_1)|) = 1$, proving Corollary~\ref{ifftheorem} from the introduction.

\section{Edge contraction} \label{sec:contractions}

Theorems \ref{T:GOQ2.8} and \ref{GCOQ} shed some light on the behavior of the Jacobian when an edge is removed from the graph. It would be natural to ask what happens when a vertex is removed. This is the subject of the following corollary and proposition resulting from our main theorem.

\begin{corollary}\label{C:2.4}
Let $G$ be graph and let $e = (x,y)$ be an edge. Let $G/e$ denote the graph obtained by identifying the vertices $x$ and $y$. Let $S$ be as defined in Theorem \ref{T:GOQ2.8}. Then $[\J(G):S]$ $\big\vert$ $gcd(|\J(G)|,|\J(G/e)|)$.
\end{corollary}
% find a reference for "well-known".
\begin{proof}
Let $T(G)$ denote the number of spanning trees of $G$. As explained in \cite{kocay}, $T(G)$ obeys the following recurrence: $T(G) = T(G-e) + T(G/e)$, where $G-e$ denotes the graph obtained from $G$ upon removal of the edge $e$. Thus $\gcd(|\J(G)|,|\J(G/e)|) = \gcd(|\J(G)|,|\J(G-e)|)$, and the result follows by Theorem \ref{T:GOQ2.8}.
\end{proof}

\begin{remark}
This corollary essentially is saying that $ T(G/(x,y)) = C_{yy}^{\tilde{L}}$ where $ \tilde{L} $ is $ L $ with the $ x $ row and column deleted. This can be seen by comparing the above recurrence with the equation for $\det \tilde{L}$ in the proof of Theorem \ref{T:GOQ2.8}.
\end{remark}

%\begin{remark}
%Using the same techniques as in the proof of Kirchoff's matrix tree theorem using the incidence matrix and Cauchy-Binet formula for the determinant, it can be shown that all the entries of the cofactor matrix are counting spanning trees in some sense. Label all the edges in $ G$ and let $ S(G / (x,y))$ be the collection of sets of $ n-2 $ labeled edges corresponding to spanning trees on $G / (x,y)$. Then 
%\begin{align*}
%C_{wy}^{\tilde{L}} = \left|S(G / (x,y)) \bigcap S(G / (x,w))  \right|.
%\end{align*}
%This is somewhat unrelated to the main goal of this paper, so we will not provide the full proof.
%\end{remark}

It would then be natural to ask whether $\J(G/e)$ is cyclic if the orders of the Jacobians of $G$ and $G/e$ are relatively prime. Theorem 5.1 of \cite{Lor1} proves that in this case $\J(G)$ is cyclic, and in Lemma 6.2 of \cite{Lor2} this result is extended to show that $\J(G-e)$ is cyclic. We now apply Theorem \ref{T:GOQ2.8} to show that on an undirected graph, $\J(G/e)$ must be cyclic as well.

\begin{proposition}
Let $G$ be a graph, and $e$, $G/e$ defined as above. Let $z$ be the vertex obtained by identifying $x$ and $y$. If $|\J(G)|$ and $|\J(G/e)|$ are relatively prime and 
\[
	D_x(v) =
    \begin{cases}
    	0 & (v,x) \notin E(G) \\
        (\text{val}_G(x)-1) & v = z \\
        -1 & (v,x) \in E(G), v \ne z,
    \end{cases}
\]
then $\J(G/e)$ is cyclic and $ [D_x] $ generates it.
\end{proposition}

\begin{proof}
Consider some arbitrary $D' \in$ Div$^0(G/e)$ and let $D \in$ Div$^0(G)$ such that $\forall v \neq x,y$, $D(v) = D'(v)$. Then since the values sum to 0, we have $D'(z) = D(x) + D(y)$. Consider the firing script $\sigma$ taking $D$ to an equivalent multiple of $\delta_{xy}$, and let $\sigma(v)$ denote the number of times the vertex $v$ is fired. Without loss of generality let $\sigma(x) = 0$. Now let $\sigma'$ be a set of firing moves on $G/e$ such that $\sigma'(z) = \sigma(y)$ and $\sigma'(v) = \sigma(v)$ otherwise. Let $D_1'$ be the divisor obtained from $D'$ after the firing of $\sigma'$. Then we have
\begin{align*}
    D_1'(v) &= D'(v) - \text{val}_{G/e}(v)\sigma'(v) + \sum_{(u,v)\in G/e}\sigma'(u) \\
    &=
    \begin{cases}
    	0 & (v,x) \notin G \\
        -\sigma'(z)(\text{val}_G(x)-1) & v = z \\
        \sigma'(z) & (v,x) \in G, v \ne z.
    \end{cases}
\end{align*}
Thus $D_1' = -\sigma'(z)D_x$, where $D_x$ is the divisor given by
\[
	D_x(v) =
    \begin{cases}
    	0 & (v,x) \notin G \\
        (\text{val}_G(x)-1) & v = z \\
        -1 & (v,x) \in G, v \ne z.
    \end{cases}
\]
Thus every $D' \in \text{Div}^0(G/e)$ is a equivalent to a multiple of $D_x$. Hence $\J(G/e)$ is generated by $[D_x]$, thus it is cyclic.
\end{proof}

Note that the form of $D_x$ depends on a choice of ordering of $x$ and $y$. The following remark gives the relation between the generators if the ordering is reversed.

\begin{remark}
Let $D_y$ be the divisor obtained by reversing the ordering of $x$ and $y$ such that $ [D_y]$ generates $\J(G)$. Then $D_x \sim -D_y$.
\end{remark}

\section{Bounding the order of $[\delta_{xy}]$ below} \label{sec:bounding}

The following proposition bounds the maximal order of some $[\delta_{xy}]$.

\begin{proposition}\label{weakbounds}
Let $G$ be a graph with $n=|V(G)|$ vertices and $\epsilon=|E(G)|$ edges. 
\begin{enumerate}
\item There exists some edge $(x,y)$ such that $|[\delta_{xy}]|_{\J(G)} \ge \epsilon/(n-1)$.
\item There exists some edge $(u,v)$ such that $|[\delta_{uv}]|_{\J(G)} \ge \epsilon/(\epsilon-n+1)$.
\end{enumerate}
\end{proposition}

Note that in the first inequality the bound is tight for a spanning tree on $n$ vertices, and in the second the bound is tight for an $n$-cycle.

\begin{proof}
Recall from Lemma \ref{Cyylem} that $C_{yy}^{\tilde{L}}$ gives exactly the number of spanning trees of $G$ containing the edge $(x,y)$. Using the pigeonhole principle, we obtain the following bound: there exists $(x,y)$ such that $C_{yy}^{\tilde{L}} \le m(n-1)/\epsilon$. By Theorem \ref{T:GOQ2.8}, we have $m/|[\delta_{xy}]|_{\J(G)} \le \gcd(m, C_{yy}^{\tilde{L}}) \le C_{yy}^{\tilde{L}} \le m(n-1)/\epsilon$. The first of the above inequalities follows.

For the second, we instead bound $C_{yy}^{\tilde{L}}$ from below: there is some $(x,y)$ such that $C_{yy}^{\tilde{L}} \ge m(n-1)/\epsilon$. Thus $\gcd(m, C_{yy}^{\tilde{L}}) \le m - C_{yy}^{\tilde{L}} \le m - m(n-1)/\epsilon$ and the second inequality follows.
\end{proof}

Note that Lemma 27 of \cite{GJRWW14} gives a lower bound for $|[\delta_{xy}]|_{\J(G)}$ in a biconnected graph as val$(x)$, and the maximum valency of any vertex is bounded from below by $2\epsilon/n$. This is almost always better than the bounds of Corollary \ref{weakbounds}, but ours hold more generally. In fact, for a biconnected graph we are able to strengthen the result of \cite{GJRWW14}.

\begin{proposition}\label{GSUMRY14}
Given a biconnected simple graph $G$ and an edge $(x,y)$,
\[
|[\delta_{xy}]|_{\J(G)} \ge val(x) + \frac{val(x)-1}{val(y)-1}.
\]
\end{proposition}

\noindent This is always stronger than the result of \cite{GJRWW14}, and equality is attained when $G$ is a complete graph. 

\begin{proof}
Let $a$ be an integer such that $a\delta_{xy} \sim 0$, and let $\sigma$ be the firing script taking $a\delta_{xy}$ to $0$. Now let $z$ be the vertex such that $z$ is fired the least times (i.e. $\sigma(z)$ is the minimum of $\sigma$). Without loss of generality by Convention \ref{notat}, we can set $\sigma(z) = 0$. Now note that since $\sigma(v) \ge 0$ for all $v$, the value of the divisor on $z$ can only increase upon firing $\sigma$. We have 3 cases:

{\setlength{\parindent}{0cm}
\underline{Case 1}: $z = y$.

This case is not possible since the initial value on $y$ is $a$. If $\sigma(y) = 0$, $\sigma$ can only increase the value of the divisor on $y$, thus in this case we cannot have that $a\delta_{xy} + L\sigma = \textbf{0}$.

\underline{Case 2}: $z \neq x,y$.

We assume $\sigma(y) > 0$, else we consider it as Case 1. $z$ does not fire, and hence for the value of the divisor to remain at $0$ we require that all neighbors of $z$ do not fire as well. We repeat this argument for all neighbors of $z$ that are not $x$ or $y$. Since $G$ is biconnected there exists a path from each vertex $z$ to $y$ not passing through $y$. Thus we eventually conclude that $\sigma(y) = 0$. This is a contradiction since we initially assumed that $\sigma(y) > 0$.

\underline{Case 3}: $z = x$.

We can assume that for all $v \neq x$, $\sigma(v) > 0$, since otherwise this would reduce to one of the earlier cases. Thus each neighbor of $x$ fires at least once. For the value on $y$ to be zero after firing, $y$ must fire at least $a/\text{val}(y)$ times. In fact, each neighbor of $y$ that is not $x$ must fire at least once (else we have Case 2), thus $y$ needs to fire at least $(a + \text{val}(y) - 1)/\text{val}(y)$ times. Since $a = \sum \sigma(v)$ over all neighbors $v$ of $x$, we have $a \ge \text{val}(x) - 1 + (a + \text{val}(y) - 1)/\text{val}(y) = \text{val}(x) + (a - 1)/\text{val}(y)$. Further manipulation gives the result.
}
\end{proof}

\begin{remark} Much of this proof follows immediately from the fact that $\sigma$ is an integer-valued harmonic function on the graph $G$ with source $y$ and sink $x$. This necessarily implies that $y$ and $x$ are the unique maximum and minimum of $\sigma$ respectively. The bound follows soon after.
\end{remark}

We have a similar result for multigraphs.

\begin{proposition}\label{GSUMRY14v2}
Given a multigraph $G$ and an edge $(x,y)$,
\[
|[\delta_{xy}]|_{\J(G)} \ge (val(x) -1)\frac{val(y)}{val(y)-1}.
\]
\end{proposition}
\begin{proof}
The proof proceeds similarly, except this time we can only bound the number of times $y$ fires by $a/\text{val}(y)$. Thus we have $a \ge \text{val}(x) - 1 + a/\text{val}(y)$. Solving this gives the result.
\end{proof}
Note that as long as $x$ is chosen to be the vertex with higher valency, this result is always stronger than that of \cite{GJRWW14}.

\section{Conjectures and experimental results on random graphs} \label{sec:random}

\subsection{Random graph conjectures}
As in the introduction, we consider Erd\H{o}s-R\'enyi random graphs. 
%Like in the introduction, we consider Erd\H{o}s-R\'enyi random graphs. 
One of the original motivations of our research was to use facts about generators to find a lower bound on the probability that the Jacobian of a graph is cyclic. As mentioned previously, it has been proven in \cite{W14} that $\prod_{i=1}^\infty \zeta(2i + 1)^{-1}$ is an upper bound on the probability. However, no lower bound is known. It is conceivable that better understanding of when $ [\delta_{xy}] $ generates the Jacobian will shed light on a lower bound.

By Corollary \ref{ifftheorem} we have that $ [\delta_{xy}] $ is a generator of $ \J(G)$ exactly when $|\J(G)| $ and $|\J(G_1)|$ are relatively prime. The question then becomes how the orders of these two groups depend on each other. Two positive integers chosen uniformly at random from the integers less than $ n $ are coprime with probability $ \zeta(2)^{-1}$ as $ n $ goes to infinity \cite{AN10}. We know that $|\J(G)| $ and $|\J(G_1)|$ are not independent. For example, if $ \J(G) $ is not cyclic, Corollary \ref{ifftheorem} tells us that the orders cannot be relatively prime. However, when $ \J(G) $ is cyclic, there is no obvious relationship between $|\J(G)| $ and $|\J(G_1)|$. So, the question becomes whether $|\J(G)| $ and $|\J(G_1)|$ behave like random integers when $\J(G) $ is cyclic and we choose an edge $ (x,y) $ at random to remove (or add if it is not an edge of the graph).  Performing various simulations led us to the following conjecture.

\begin{conjecture}\label{fixconj}
Let $ p $ be fixed and $ G_{n,p} $ be the Erd\H{o}s-R\'enyi random graph. For each graph, fix vertices $ x $ and $ y$. Then
\begin{align*}
\lim_{n\to \infty} \mathbb{P}\left([\delta_{xy}] \text{ generates } \J(G_{n,p}) \ |\ \J(G_{n,p}) \text{ cyclic}\right) = \zeta(2)^{-1} \approx 0.607927
\end{align*}
\end{conjecture}

If this is true and there is a reasonable amount of independence across choices of $ x $ and $ y$, we would expect to be able to find some edge $ x,y $ that allows us to make $ |\J(G)|$ and $ |\J(G_1)| $ coprime. This leads us to the following conjecture.

\begin{conjecture}\label{existconj}
Let $ p $ be fixed and $ G_{n,p} $ be the Erd\H{o}s-R\'enyi random graph. Then
\begin{align*}
\lim_{n\to \infty} \mathbb{P} \left(\exists\ [\delta_{xy}] \text{ generator of } \J(G_{n,p}) \ |\  \J(G_{n,p}) \text{ cyclic} \right) = 1
\end{align*}
Or equivalently, that
\begin{align*}
\lim_{n\to \infty} \mathbb{P} \left(\exists\ [\delta_{xy}] \text{ generator of } \J(G_{n,p})  \right) =  \lim_{n\to \infty} \mathbb{P} \left(\J(G_{n,p}) \text{ cyclic} \right)
\end{align*}
\end{conjecture}

\subsection{Experimental data}
To provide some support beyond intuition for Conjectures \ref{fixconj} and \ref{existconj}, we conducted simulations using SageMath \cite{sage} for various values of $ n$ (available on GitHub at \cite{git}). For each $ n $ we conducted 100,000 trials where each trial consisted of generating $ G_{n,1/2} $ with a cyclic Jacobian and checking the desired property. 
\begin{center}
\begin{tabular}{|c|c|c|}
\hline
n & $\mathbb{P}($ fixed $[\delta_{xy}]$ generates $\J(G)$) & $\mathbb{P}( \exists\  [\delta_{xy}]$ that generates $\J(G)$) \\ \hline
5 & 0.75561 & 1.00000 \\ \hline
10 & 0.58766 & 0.99890 \\ \hline
20 & 0.60732 & 1.00000 \\ \hline
40 & 0.60793 & 1.00000 \\ \hline
\end{tabular}
\end{center}
We believe that these data support both of the above conjectures.

\subsection{Conjecture on the order of $ [\delta_{xy}]$}

In Section ~\ref{sec:bounding}, we proved some lower bounds on the order of some $ [\delta_{xy}] $ in the graph. However, by experimentation these do not seem to be the best possible lower bounds. We conjecture that a better bound can be found if we assert that a graph is biconnected. Here biconnected means that any vertex can be removed while leaving a connected graph on $ n -1 $ vertices, which is equivalent to saying the graph cannot be decomposed into the wedge product of two proper subgraphs that are not just vertices. This condition prevents the construction of pathological counterexamples using wedge products and gives, based on computations of Dhruv Ranganathan and Jeffrey Yu \cite{dhruv}, the following conjecture.

\begin{conjecture}
Let $ G $ be a biconnected, simple graph with $ n $ vertices. Fix a vertex $ x $. Then there exists a vertex $ y $ such that $ |[\delta_{xy}]|_{\J(G)} \geq n$. 
\end{conjecture}

We further tested this conjecture on random graphs of varying size and all biconnected graphs on small numbers of vertices, and we could not find any counterexamples. An affirmative resolution of this conjecture could provide some leverage over finding classes of graphs with cyclic Jacobian. Moreover, following the work of \cite{GJRWW14}, this conjecture would prove that for any positive integer $ n $ there exists an integer $ k_n $ such that for all $ k > k_n $ there is no biconnected graph $ G $ with $ \J(G) \cong (\mathbb{Z}/n\mathbb{Z})^k$.

%%%%%%%%%%%%%%%%%%%%%%%%%%%%%%%%%%%%%%%%%%%%%%%%
%%%%%%%%%%%%%%%%%%%%%%%%%%%%%%%%%%%%%%%%%%%%%%%%
%%%%%%%%%%%%%%%%%%%%%%%%%%%%%%%%%%%%%%%%%%%%%%%%
%%%%%%%%%%%%%%%%%%%%%%%%%%%%%%%%%%%%%%%%%%%%%%%%
%%%%%%%%%%%%%%%%%%%%%%%%%%%%%%%%%%%%%%%%%%%%%%%%
%%%%%%%%%%%%%%%%%%%%%%%%%%%%%%%%%%%%%%%%%%%%%%%%
%%%%%%%%%%%%%%%%%%%%%%%%%%%%%%%%%%%%%%%%%%%%%%%%
%%%%%%%%%%%%%%%%%%%%%%%%%%%%%%%%%%%%%%%%%%%%%%%%
%%%%%%%%%%%%%%%%%%%%%%%%%%%%%%%%%%%%%%%%%%%%%%%%
%%%%%%%%%%%%%%%%%%%%%%%%%%%%%%%%%%%%%%%%%%%%%%%%
%%%%%%%%%%%%%%%%%%%%%%%%%%%%%%%%%%%%%%%%%%%%%%%%
%%%%%%%%%%%%%%%%%%%%%%%%%%%%%%%%%%%%%%%%%%%%%%%%
%%%%%%%%%%%%%%%%%%%%%%%%%%%%%%%%%%%%%%%%%%%%%%%%
%%%%%%%%%%%%%%%%%%%%%%%%%%%%%%%%%%%%%%%%%%%%%%%%
%%%%%%%%%%%%%%%%%%%%%%%%%%%%%%%%%%%%%%%%%%%%%%%%

\bibliographystyle{siam}
\bibliography{generatorsofjacobians}

{\small {\sc Yale University, New Haven, CT 06520} \par
\textit{Email address, David Brandfonbrener}: {\tt \href{mailto:david.brandfonbrener@yale.edu}{david.brandfonbrener@yale.edu}} \par
\textit{Email address, Pat Devlin}: {\tt \href{mailto:patrick.devlin@yale.edu}{patrick.devlin@yale.edu}} \par
\textit{Email address, Netanel Friedenberg}: {\tt \href{mailto:netanel.friedenberg@yale.edu}{netanel.friedenberg@yale.edu}} \par
\textit{Email address, Yuxuan Ke}: {\tt \href{mailto:yuxuan.ke@yale.edu}{yuxuan.ke@yale.edu}} \par
\textit{Email address, Henry Reichard}: {\tt \href{mailto:henry.reichard@yale.edu}{henry.reichard@yale.edu}} \par
\textit{Email address, Ethan Sciamma}: {\tt \href{mailto:ethan.sciamma@yale.edu}{ethan.sciamma@yale.edu}} \\

{\sc The College of New Jersey, Ewing, NJ 08628} \par
\textit{Email address, Steffen Marcus}: {\tt \href{mailto:marcuss@tcnj.edu}{marcuss@tcnj.edu}} \\
\end{document}